\documentclass[11pt]{article}
\usepackage{graphicx} 
\usepackage{graphicx}%
\usepackage{multirow}%
\usepackage{amsmath,amssymb,amsfonts}%
\usepackage{amsthm}%
\usepackage{mathrsfs}%
\usepackage[title]{appendix}%
\usepackage{xcolor}%
\usepackage{textcomp}%
\usepackage{manyfoot}%
\usepackage{booktabs}%
\usepackage{algorithm}%
\usepackage{algorithmicx}%
\usepackage{algpseudocode}%
\usepackage{listings}%
\usepackage{bm}
\usepackage[utf8]{inputenc}
\usepackage[T1]{fontenc}

\newcommand{\mb}{\mathbf}

\newcommand{\g}{\gamma}
\newcommand{\al}{\alpha}

\newcommand{\FD}{\mathcal{F}_D}
\newcommand{\bp}{\theta} 
\newcommand{\dl}{\mb{d}_{\mb{1}}} 
\newcommand{\dr}{\mb{d}_{\mb{n}}} 

\newcommand{\pe}{\varepsilon}
\newcommand{\bs}{\boldsymbol}
\newcommand{\FDe}{\mathcal{F}_D^{\varepsilon}}
\newcommand{\pep}{\xi}
\newcommand{\pepp}{\eta}
\newcommand{\FDp}{\mathcal{F}_D^{\pep}}
\newcommand{\FDpp}{\mathcal{F}_D^{\pepp}}
\newcommand{\FDpL}{{\hat{\mathcal{F}}}_D^{\pep}}
\newcommand{\FDpLB}{{\hat{\mathcal{F}}}_{D,B}^{\pep}}

\newcommand{\FN}{\mathcal{F}_N}

\newcommand{\tFD}{\tilde{\mathcal{F}}_D}
\newcommand{\tFN}{\tilde{\mathcal{F}}_N}

\newcommand{\tFDx}{\tilde{\mathcal{F}}_{Dx}}
\newcommand{\tFDy}{\tilde{\mathcal{F}}_{Dy}}

\newcommand{\rev}[1]{{\color{black}#1}}

\newcommand{\vertiii}[1]{{\left\vert\kern-0.25ex\left\vert\kern-0.25ex\left\vert #1 
    \right\vert\kern-0.25ex\right\vert\kern-0.25ex\right\vert}}
    
\newtheorem{theorem}{Theorem}
\newtheorem{definition}{Definition}
\newtheorem{lemma}{Lemma}
\newtheorem{remark}{Remark}

\begin{document}
\title{Stable and high-order accurate finite difference methods for the diffusive viscous wave equation}
\author{Siyang Wang\thanks{Department of Mathematics and Mathematical Statistics, {Ume\aa} University, Sweden (siyang.wang@umu.se)}}

\maketitle
\begin{abstract}
The diffusive viscous wave equation describes wave propagation in diffusive and viscous media. Examples include seismic waves traveling through the Earth's crust, taking into account of both the elastic properties of rocks and the dissipative effects due to internal friction and viscosity; acoustic waves propagating through biological tissues, where both elastic and viscous effects play a significant role. We propose a stable and high-order finite difference method for solving the governing equations. By designing the spatial discretization with the summation-by-parts property, we prove stability by deriving a discrete energy estimate. In addition, we derive error estimates for problems with constant coefficients using the normal mode analysis and for problems with variable coefficients using the energy method. Numerical examples are presented to demonstrate the stability and accuracy properties of the developed method.
\end{abstract}




\section{Introduction}

Propagation of waves presents a complex challenge in domains characterized by vast expanses and material heterogeneity, where waves travel considerable distances relative to their wavelength. Analytical solutions to the governing partial differential equations are generally unattainable in realistic models, necessitating the utilization of numerical methods for approximated solutions to study the underlying wave phenomena. This demands the development, analysis and implementation of robust and accurate numerical techniques.

Classical dispersion analysis \cite{Hagstrom2012,Kreiss1972} demonstrated that high-order accurate numerical methods solve wave propagation problems more efficiently compared to low-order methods when solutions have sufficient regularity. Moreover, ensuring the stability of numerical methods is crucial to prevent any undesirable numerical artifacts. Over time, numerous computational techniques have been developed to address the challenges inherent in wave propagation problems.

In the finite difference framework, while constructing high-order accurate finite difference stencils using Taylor series is straightforward, maintaining stability for initial-boundary-value problems was challenging. However, this challenge has been largely addressed through the utilization of finite difference stencils possessing a summation-by-parts (SBP) property \cite{Kreiss1974}. The SBP property effectively emulates the integration-by-parts principle discretely. When boundary conditions are appropriately imposed, a discrete energy estimate analogous to the continuous energy estimate of the governing equation can be derived, thereby ensuring stability. Techniques for properly imposing boundary conditions include  the simultaneous-approximation-term (SAT) method \cite{Carpenter1994}, the projection method \cite{Mattsson2018} and the ghost point method \cite{Sjogreen2012}. In the past two decades, SBP finite difference methods have been extensively developed for wave propagation problems. For the acoustic wave equation, the SAT method was developed to impose boundary conditions \cite{Mattsson2009} and material interface conditions \cite{Mattsson2008}, and then extended to domains with curved boundaries and interfaces \cite{Virta2014}. The methodology was further developed for discretizations with different mesh sizes across interfaces \cite{Almquist2019,Wang2018,Wang2016}. For the elastic wave equation, see  \cite{Almquist2021,Duru2014V,Virta2015}, and also the two review articles \cite{Del2014Review,Svard2014}.

Numerical methods for the diffusive viscous wave equation have also been developed, for example the finite volume method \cite{Mensah2019}, the second order accurate finite difference method \cite{Zhao2014,Zhao2014b}, and more recently a local discontinuous Galerkin (LDG) method \cite{Ling2023}. In the LDG method, high-order accuracy is achieved through the utilization of high-order local basis functions. To bridge this advancement in the finite difference framework, we develop a stable and high-order accurate SBP-SAT finite difference method for the diffusive viscous wave equation with Dirichlet or Neumann boundary conditions.  Stability of the developed method is rigorously established through a derivation of a discrete energy estimate. Additionally, we derive a priori error estimates by employing normal mode analysis for problems with constant coefficients, and the energy method for problems with variable coefficients. 

The remainder of this paper is structured as follows. In Sec. 2, we introduce the governing equation, and derive a continuous energy estimate. Sec. 3 begins with a presentation of the SBP concepts, followed by the construction of the  semidiscretization by using the SBP-SAT approach. Moreover, we prove stability of the semidiscretization by deriving a discrete energy estimate. A priori error estimates are then established in Sec. 4. We present numerical examples in Sec. 5 and draw conclusion in Sec. 6.

\section{The governing equation}
The diffusive viscous wave equation in two space dimension is 
\begin{equation}\label{DVWE}
    u_{tt} + \alpha u_t - (\nabla\cdot(\beta^2\nabla u))_t - \nabla\cdot (\g^2\nabla u) = f,\quad (x,y)\in\Omega, \quad t\in (0,T].
\end{equation}
The material properties are described by the diffusive attenuation parameter $\al\geq 0$, the viscous attenuation parameter $\beta\geq 0$, and the wave speed $\gamma> 0$ in non-dispersive media. All three parameters may vary smoothly in space. Equation \eqref{DVWE} is equipped with initial conditions 
\[
u(x,y,0) = g_1(x,y),\quad u_t(x,y,0) = g_2(x,y).
\]
In addition, suitable boundary conditions must be provided so that the model problem is wellposed. \rev{Wellposedness of the diffusive viscous wave equation was thoroughly analyzed in \cite{Han2020}. Below, we follow the procedure and use the energy method to analyze stability for different boundary conditions.} 

We multiply equation \eqref{DVWE} by $u_t$, and integrate in $\Omega$,
\begin{align*}
    \int_{\Omega} u_ t f &= \int_{\Omega} u_t u_{tt} + \int_{\Omega} u_t \alpha u_t - \int_{\Omega} u_t (\nabla\cdot(\beta^2\nabla u))_t -
 \int_{\Omega} u_t \nabla\cdot (\g^2\nabla u) \\
 =&\frac{1}{2}\frac{d}{dt} \|u_t\|^2_{L^2(\Omega)} +  \|\sqrt{\alpha}u_t\|^2_{L^2(\Omega)} + \int_{\Omega} \nabla u_t \cdot \beta^2\nabla u_t - \int_{\partial\Omega} u_t \beta^2 \nabla u_t \cdot \mb{n} \\
 &+ \int_{\Omega} \nabla u_t \cdot \gamma^2\nabla u - \int_{\partial\Omega} u_t \gamma^2\nabla u  \cdot \mb{n} \\
 =&\frac{1}{2}\frac{d}{dt} \|u_t\|^2_{L^2(\Omega)} + \|\sqrt{\alpha}u_t\|^2_{L^2(\Omega)} +   \|{\beta}\nabla u_t\|^2_{L^2(\Omega)} + \frac{1}{2}\frac{d}{dt} \|\gamma \nabla u\|^2_{L^2(\Omega)} \\
 &-\int_{\partial\Omega}u_t (\beta^2 \nabla u_t \cdot \mb{n} + \gamma^2 \nabla u \cdot \mb{n}). 
\end{align*}
We have 
\begin{align}
&\frac{d}{dt} \left(\frac{1}{2}\|u_t\|^2_{L^2(\Omega)} + \frac{1}{2}\|\gamma \nabla u\|^2_{L^2(\Omega)}\right) \label{ContinuousEnergy1D}\\
=& -\|\sqrt{\alpha}u_t\|^2_{L^2(\Omega)} -\|{\beta}\nabla u_t\|^2_{L^2(\Omega)} + \int_{\partial\Omega}u_t (\beta^2 \nabla u_t \cdot \mb{n} + \gamma^2 \nabla u \cdot \mb{n}) + \int_{\Omega} u_ t f.  \notag
\end{align}

The left-hand side is the energy change rate. On the right-hand side, the first two terms correspond to diffusive and viscous attenuation that dissipates energy. The third term is the boundary contribution, and it vanishes with homogeneous Dirichlet boundary condition
\begin{equation}\label{DirichletBC}
    u=0,\quad (x,y)\in\partial\Omega,
\end{equation}
or homogeneous Neumann boundary condition, 
\begin{equation}\label{NeumannBC}
    \nabla u\cdot \mb{n}=0,\quad  (x,y)\in\partial\Omega.
\end{equation}
Therefore, both boundary conditions lead to a wellposed problem. In the special case when $\alpha=\beta=0$, the governing equation becomes the standard acoustic wave equation, and conserves energy when the forcing $f$ vanishes. 

\section{The SBP finite difference method}
\subsection{The SBP finite difference operators}\label{sec_SBP_1D}
Consider a uniform grid $\mb{x}=[x_1, x_2, \cdots, x_n]^T$ in one dimensional domain $[0,1]$, with grid points $x_j$ and grid spacing $h$ defined as
\begin{equation}\label{grid1D}
    x_j = (j-1)h,\quad j=1,2,\cdots,n,\quad h = \frac{1}{n-1}. 
\end{equation}
We define $u_j :=u(x_j)$, $v_j :=v(x_j)$,  where $u(x),v(x)$ are real-valued functions in $L^2([0,1])$. Let
\[
\mb{u}=[u_1,u_2,\cdots,u_n]^T, \quad \mb{v}=[v_1,v_2,\cdots,v_n]^T,
\]
denote the corresponding real-valued grid functions on $\mb{x}$.

The SBP concept was first introduced by Kreiss and Scherer in \cite{Kreiss1974}, which is based on a weighted inner product 
\begin{equation}\label{Hinner}
(\mb{u}, \mb{v})_H = h\sum_{i=1}^n \omega_i u_i v_i,    
\end{equation}
where all weights $\omega_i$ are positive and do not depend on $h$. In matrix-vector notation, we write $(\mb{u}, \mb{v})_H=\mb{u}^T H \mb{v}$, where $H$ is a diagonal matrix with entries $H_{ii}=h \omega_i>0$. The inner product \eqref{Hinner} defines the induced SBP norm $\|\cdot\|_H$. The weights $\omega_i=1$ in the interior of the domain, and $\omega_i\neq 1$ on a few grid points near each boundary,
\begin{align*}
&w_i=1,\quad i \in \mathcal{N}_I=\{k+1, \cdots, n-k\},\\
&w_i\neq 1, \quad i \in \mathcal{N}_B=\{1,\cdots,k\}\cup\{n-k+1,\cdots, n\}.
\end{align*}
The value $k$ depends on $p$ but not $n$.

In \cite{Kreiss1974}, the SBP operator for the approximation of the first derivative, $D_1\approx d/dx$, was constructed. The operator $D_1$ satisfies the following SBP identity.
\begin{definition}[first derivative SBP identity]
The first derivative SBP finite difference operator $D_1$ satisfies 
\begin{equation}\label{SBPidentityD1}
(\mb{u}, D_1\mb{v})_H = -(D_1 \mb{u}, \mb{v})_H - u_1v_1 + u_nv_n,    
\end{equation}
for all grid functions $\mb{u},\mb{v}$. 
\end{definition}
The SBP norm $H$ is also a quadrature \cite{Hicken2016}. Then, \eqref{SBPidentityD1} is a discrete analogue of the integration-by-parts formula, 
\[
\int_0^1 u v_x dx = - \int_0^1 u_x v dx - u(0)v(0) + u(1)v(1). 
\]

Central finite difference stencils are used in $D_1$ on the grid points in the interior $\mathcal{N}_I$, where the weights in the SBP norm equal to one. On the first $k$ grid points near each boundary, special one-sided boundary closures are employed so that the SBP identity \eqref{SBPidentityD1} holds, and this limits the order of accuracy. If the central finite difference stencils in the interior have truncation error $\mathcal{O}(h^{2p})$, then the truncation error of the boundary closures can at best be $\mathcal{O}(h^{p})$. Such SBP operators are constructed in \cite{Kreiss1974} for $p=1,2,3,4$ with $k=p$. We note that the number of grid points with boundary closure, $k$, depends on the order of accuracy $p$ but not the total number of grid points $n$. In \cite{Mattsson2014}, optimized SBP operators were constructed based on diagonal norms. The truncation error of the boundary closure remains to be $\mathcal{O}(h^{p})$, but the coefficient of the leading order term is significantly smaller than traditional SBP operators, resulting in improved accuracy. We also note that the truncation error of the boundary closures can be $\mathcal{O}(h^{2p-1})$ by using a non-diagonal SBP norm. For problems with variable coefficients, the non-diagonal SBP norm matrix does not commute with the coefficient matrix, making it difficult to establish stability estimate for the discretization \cite{Mattsson2013}. Therefore, SBP operators based on non-diagonal norm are not widely used.  

SBP operators for the second derivative, $D_2^{(b)}\approx \frac{d}{dx}\left(b(x)\frac{d}{dx}\right)$, were constructed in \cite{Mattsson2012}. The variable coefficient $b(x)\geq 0$ models material property.  
The operator $D_2^{(b)}$ satisfies the following SBP identity.
\begin{definition}[second derivative SBP identity]
The second derivative SBP finite difference operator $D_2^{(b)}$ satisfies 
\begin{equation}\label{SBPidentityD2}
(\mb{u}, D_2^{(b)}\mb{v})_H = -(\mb{u}, \mb{v})_{A^{(b)}} - b_1 u_1 \mb{d}_\mb{1}^T \mb{v} + b_n u_n \mb{d}_\mb{n}^T \mb{v},    
\end{equation}
for all grid functions $\mb{u},\mb{v}$. Here, the matrix $A^{(b)}$ is symmetric positive semidefinite, the boundary difference formula $\mb{d}_\mb{1}^T$ and $\mb{d}_\mb{n}^T$ approximate the first derivative at $x_1$ and $x_n$, respectively. The coefficients on the boundaries are $b_1=b(x_1)$ and $b_n=b(x_n)$.  
\end{definition}
The SBP identity \eqref{SBPidentityD2} is a discrete analogue of the integration-by-parts formula,
\[
\int_0^1 u (bv_x)_x dx = -\int_0^1 b u_x v_x dx - b(0)u(0)v_x(0) + b(1)u(1)v_x(1).
\]
The accuracy property of $D_2^{(b)}$ is similar to $D_1$, that is, the truncation error is $\mathcal{O}(h^{2p})$ on the grid points in the interior, and $\mathcal{O}(h^{p})$ on the first $k$ grid points near each boundary, for $p=1,2,3$. \rev{The boundary derivative approximations $\mb{d}_\mb{1}^T$ and $\mb{d}_\mb{n}^T$ have truncation error $\mathcal{O}(h^{p+1})$, so they are different from the stencils in $D_1$ on the boundary. } 

\rev{The SBP operators $D^{(b)}_2$ constructed in \cite{Mattsson2012} are compatible with $D_1$, meaning that $A^{(b)}$ can be decomposed as
\begin{equation}\label{Ab}
    A^{(b)} = D_1^T H \Lambda_b D_1 + R^{(b)},
\end{equation}
where $\Lambda_b$ is a diagonal matrix with $(\Lambda_b)_{ii} = b(x_i)$, and $R^{(b)}$ is symmetric positive semidefinite.  

By replacing the boundary derivative approximations $\mb{d}_\mb{1}^T$ and $\mb{d}_{\mb{n}}^T$ with the stencils in $D_1$, we obtain the fully-compatible \cite{MattssonParisi2010} second derivative SBP operator $\hat{D}_2^{(b)}$ that satisfies
\begin{equation}\label{SBPidentityFC}
(\mb{u}, \hat{D}_2^{(b)}\mb{v})_H = -(\mb{u}, \mb{v})_{A^{(b)}} - b_1 u_1 \hat{\mb{d}}_{\mb{1}}^{T} \mb{v} + b_n u_n \hat{\mb{d}}_{\mb{n}}^{T} \mb{v}.    
\end{equation}
In \eqref{SBPidentityFC}, the boundary derivative approximations $\hat{\mb{d}}_{\mb{1}}^{\mb{T}}$ and $\hat{\mb{d}}_{\mb{n}}^{\mb{T}}$ are exactly the same as the stencils in $D_1$ on the left and right boundary, respectively. Comparing with the original SBP operator $D_2^{(b)}$, using the fully-compatible version $\hat{D}_2^{(b)}$ simplifies significantly stability analysis of certain problems \cite{Duru2014V,MattssonParisi2010,Almquist2020,Duru2022}. The drawback of using $\hat{D}_2^{(b)}$ is that its truncation error is $\mathcal{O}(h^{p-1})$ on the boundary instead of $\mathcal{O}(h^{p})$ of $D_2^{(b)}$.  
}

In \cite{Virta2014}, it was shown that operator $A^{(b)}$ can be expressed as 
\begin{equation}\label{estimateA}
\mb{v}^T A^{(b)} \mb{v} = h\bp b_{l, \min} (\dl^T \mb{v})^2 + h\bp b_{r, \min} (\dr^T \mb{v})^2+\mb{v}^T \tilde A^{(b)} \mb{v}
\end{equation}
where $b_{l, \min}$ and $b_{r, \min}$ are the smallest of $b(x)$ evaluated on the first and last $m$ grid points near each boundary,
\[b_{l, \min} = \min(b(x_1), \cdots,b(x_m)), \quad b_{r, \min} = \min(b(x_{n-m+1}),\cdots,b(x_n)).\]
The parameter $\bp>0$ is chosen as large as possible when $\tilde A^{(b)}$ is symmetric positive semidefinite. As an example, when $p=2$ we have $m=4$ and $\theta=0.2505765857$, see \cite{Virta2014}. In addition, $A^{(b)}$ has exactly one zero eigenvalue \cite{Eriksson2021}. We note that $A^{(b)}$ is an analogue of the stiffness matrix in a finite element discretization, $\mb{v}^T A^{(b)} \mb{v} \approx \int_{x_1}^{x_n} v_x^2 dx$ and \eqref{estimateA} is an analogue of the inverse inequality because 
\begin{equation}\label{InverseInequality}
\mb{v}^T A^{(b)} \mb{v} \geq h\bp (b_{l, \min} (\dl^T \mb{v})^2 + b_{r, \min} (\dr^T \mb{v})^2).
\end{equation}
\rev{The right-hand side of \eqref{InverseInequality} is used to derive stability analysis for the SBP-SAT discretization of certain problems, e.g., the wave equation with Dirichlet boundary condition and interface conditions, and the diffusive viscous wave equation. }

\rev{When the variable coefficient $b$ takes value zero on some grid points close to the boundary but not zero on the boundary, then the right-hand side of \eqref{InverseInequality} becomes zero and the aforementioned stability analysis fails. In this case, the fully-compatible SBP operator $\hat{D}_2^{(b)}$ can be used in the discretization to establish stability analysis, because a similar estimate to \eqref{InverseInequality} holds
\begin{equation}\label{InverseInequality2}
\mb{v}^T A^{(b)} \mb{v} \geq h\omega_1 (b_{1} (\hat{\mb{d}}_{\mb{1}}^{T} \mb{v})^2 + b_{n} (\hat{\mb{d}}_{\mb{n}}^{{T}} \mb{v})^2),
\end{equation}
where $\omega_1$ is the first weight in the SBP inner product \eqref{Hinner}. 
}

In the literature, it is common to refer to the accuracy of the SBP operators by its interior truncation error, i.e., $2p^{th}$ order accurate. We use this convention in this paper, and make the convergence rate of the overall discretization precise. 

Boundary conditions are not built into the SBP operators, and can be imposed either strongly by using the projection method \cite{Mattsson2018} or weakly by the SAT method \cite{Carpenter1994}. In this work, we take the latter approach. We remark that a different second derivative SBP operator was constructed in \cite{Sjogreen2012} by using ghost points. 

\subsection{An SBP-SAT discretization for the Dirichlet problem with viscosity}
Consider the governing equation \eqref{DVWE} \rev{with $\beta>0$} in a bounded one space dimensional domain $\Omega=[0,1]$ with Dirichlet boundary condition \eqref{DirichletBC}. We discretize the equation in space on the grid defined in \eqref{grid1D}, and denote the numerical solution $v_j(t)\approx u(x_j, t)$ and $\mb{v}=[v_1, v_2, \cdots, v_n]^T$. Then, the semidiscretization reads 
\begin{equation}\label{semiD1D}
    \mb v_{tt} + \Lambda_{\alpha}\mb{v}_t - D_{xx}^{(\beta^2)}\mb{v}_t - D_{xx}^{(\gamma^2)}\mb{v} + \FD = \mb{f} ,
\end{equation}
where $\Lambda_{\alpha}$ is a diagonal matrix with $(\Lambda_{\alpha})_{jj}=\alpha(x_j)$, and the term $\FD$ imposes weakly the Dirichlet boundary condition in such a way that a discrete energy estimate can be derived.  

In the following, we analyze the stability of \eqref{semiD1D} and construct $\FD$ so that a discrete energy estimate can be obtained. The forcing term $\mb{f}$ does not affect stability, and we consider zero forcing for simplified notation. We multiply equation \eqref{semiD1D} by $\mb{v}_t^T H$, and use the SBP identity \eqref{SBPidentityD2} to obtain 
\begin{align*}
    0 =& \mb{v}_t^T H \mb v_{tt} + \mb{v}_t^T H \Lambda_{\alpha}\mb{v}_t - \mb{v}_t^T H D_{xx}^{(\beta^2)}\mb{v}_t - \mb{v}_t^T H D_{xx}^{(\gamma^2)}\mb{v} + \mb{v}_t^T H \FD\\
    =&\frac{1}{2}\frac{d}{dt}\|\mb{v}_t\|_H^2 + \|\mb{v}_t\|^2_{H\Lambda_{\alpha}} - \mb{v}_t^T (-A^{(\beta^2)} -\beta_1^2 \mb{e_1}\dl^T + \beta_n^2 \mb{e_n}\dr^T  ) \mb{v}_t \\
    &-\mb{v}_t^T (-A^{(\gamma^2)} -\gamma_1^2 \mb{e_1}\dl^T + \gamma_n^2 \mb{e_n}\dr^T) \mb{v} + \mb{v}_t^T H \FD,
\end{align*}
where $\mb{e}_1 = [1,0,\cdots,0]^T$ and $\mb{e}_n = [0,\cdots,0,1]^T$. Rearranging terms yields
\begin{align}
&\frac{d}{dt}\left(\frac{1}{2}\|\mb{v}_t\|_H^2+\frac{1}{2}\|\mb{v}\|_{A^{(\gamma^2)}}^2\right) \label{DiscreteEnergy1D}\\
=& - \|\mb{v}_t\|_{H\Lambda_{\alpha}}^2 - \|\mb{v}_t\|_{A^{(\beta^2)}}^2 + \mb{v}_t^T \left(-\beta_1^2 \mb{e_1}\dl^T + \beta_n^2 \mb{e_n}\dr^T \right)\mb{v}_t\notag\\ 
&+ \mb{v}_t^T \left(-\gamma_1^2 \mb{e_1}\dl^T + \gamma_n^2 \mb{e_n}\dr^T \right)\mb{v}  -  \mb{v}_t^T H \FD.\notag
\end{align}
\rev{We construct $\FD$ so that $\FD$ is consistent with the boundary condition, and \eqref{DiscreteEnergy1D} is a discrete analogue of the continuous energy estimate \eqref{ContinuousEnergy1D}, i.e., the discrete energy change rate is nonpositive. To achieve this, we first note that on the right-hand side of \eqref{DiscreteEnergy1D}, though the first two terms are nonpositive, the third and the fourth terms are not. In addition, the boundary difference operators in third and the fourth terms make them nonsymmetric, and these two terms should be symmetrized by $\FD$. }

To this end, we make the ansatz
\begin{align}
\FD =& -H^{-1} (-\beta_1^2 \mb{e_1}\dl^T + \beta_n^2 \mb{e_n}\dr^T)^T \mb{v}_t +  H^{-1} \left(\frac{\tau_1}{h}\mb{e}_1\mb{e}_1^T + \frac{\tau_2}{h} \mb{e}_n\mb{e}_n^T \right)\mb{v}_t\notag \\
    & -H^{-1} (-\gamma_1^2 \mb{e_1}\dl^T + \gamma_n^2 \mb{e_n}\dr^T)^T \mb{v} + H^{-1} \left(\frac{\tau_3}{h}\mb{e}_1\mb{e}_1^T + \frac{\tau_4}{h} \mb{e}_n\mb{e}_n^T \right)\mb{v}. \label{FD}
\end{align}
\rev{On the right-hand side of \eqref{FD}, the first and third term symmetrize the third and fourth term in \eqref{DiscreteEnergy1D}, respectively. The second and fourth terms in \eqref{FD} are penalty terms, with penalty parameters $\tau_1,\tau_2,\tau_3,\tau_4$ to be determined so that a discrete energy estimate can be obtained. For consistency, we have $\FD\approx\mathbf{0}$ because $\mb{v}\approx 0,\mb{v}_t\approx 0$ on the boundary with the homogeneous Dirichlet boundary condition. The generalization to inhomogeneous Dirichlet boundary condition $u=g$ is straightforward, as we penalize $\mb{v}-\mb{g}$ and $\mb{v}_t - \mb{g}_t$ in \eqref{FD} instead, and $\FD\approx\mathbf{0}$ still holds. We remark that the factor $1/h$ in the penalty terms makes sure that each term in $\FD$ scales with $h$ to the same order $h^{-2}$ as the second derivative approximation. }

Substituting the ansatz of $\FD$ to \eqref{DiscreteEnergy1D} and using \eqref{estimateA}, we obtain 
\begin{align}
&\frac{d}{dt}\left(\frac{1}{2}\|\mb{v}_t\|_H^2+\frac{1}{2}\|\mb{v}\|_{A^{(\gamma^2)}}^2\right) \notag\\
=& - \|\mb{v}_t\|_{H\Lambda_{\alpha}}^2 - \|\mb{v}_t\|_{\tilde A^{(\beta^2)}}^2 \label{EA1D_s1}\\
&-h\theta\beta_{l,\min}^2 (\dl^T\mb{v}_t)^2  -2\beta_1^2 \mb{v}_t^T\mb{e_1}\dl^T\mb{v}_t - \frac{\tau_1}{h} \mb{v}_t^T\mb{e_1}\mb{e}_\mb{1}^T\mb{v}_t \label{EA1D_s2}\\
&- h\theta\beta_{r,\min}^2 (\dr^T\mb{v}_t)^2 + 2\beta_n^2 \mb{v}_t^T\mb{e_n}\dr^T\mb{v}_t - \frac{\tau_2}{h} \mb{v}_t^T\mb{e_n}\mb{e}_\mb{n}^T\mb{v}_t \label{EA1D_s3}\\ 
&-\gamma_1^2 \mb{v}_t^T \mb{e_1}\dl^T\mb{v} - \gamma_1^2 \mb{v}_t^T  (\mb{e_1}\dl^T)^T \mb{v} -\frac{\tau_3}{h}\mb{v}_t^T \mb{e}_1\mb{e}_1^T \mb{v}  \label{EA1D_s4}\\
&+ \gamma_n^2\mb{v}_t^T \mb{e_n}\dr^T\mb{v}  + \gamma_n^2 \mb{v}_t^T (\mb{e_n}\dr^T)^T \mb{v}-\frac{\tau_4}{h}\mb{v}_t^T \mb{e}_n\mb{e}_n^T \mb{v} \label{EA1D_s5}
\end{align}
It is obvious that the terms in \eqref{EA1D_s1} are nonpositive. For  \eqref{EA1D_s2}, we write 
\begin{align*}
&-h\theta\beta_{l,\min}^2 (\dl^T\mb{v}_t)^2  -2\beta_1^2 \mb{v}_t^T\mb{e_1}\dl^T\mb{v}_t - \frac{\tau_1}{h} \mb{v}_t^T\mb{e_1}\mb{e}_\mb{1}^T\mb{v}_t \\
=& -\begin{bmatrix}
     \mb{e}_{\mb{1}}^T \mb{v}_t\\
    \dl^T \mb{v}_t
\end{bmatrix}^T
\begin{bmatrix}
    \frac{\tau_1}{h} & \beta_1^2\\
    \beta_1^2 & h\theta\beta_{l,\min}^2
\end{bmatrix}
\begin{bmatrix}
     \mb{e}_{\mb{1}}^T \mb{v}_t\\
    \dl^T \mb{v}_t
\end{bmatrix}.
\end{align*}
If
\begin{equation}\label{tau1}
\left(\frac{\tau_1}{h}\right)(h\theta\beta_{l,\min}^2)-(\beta_1^2)(\beta_1^2)\geq 0 \Rightarrow \tau_1\geq \frac{\beta_1^4}{\theta\beta_{l,\min}^2}, 
\end{equation}
then the 2-by-2 matrix is symmetric positive semidefinite and \eqref{EA1D_s2} is nonpositive.  Similarly, \eqref{EA1D_s3} is nonpositive if 
\begin{equation}\label{tau2}
 \tau_2\geq \frac{\beta_n^4}{\theta\beta_{r,\min}^2}. 
\end{equation}

The terms in \eqref{EA1D_s4}-\eqref{EA1D_s5} contain both $\mb{v}$ and $\mb{v}_t$. Therefore, we need to write them as time derivatives included in the discrete energy on the left-hand side. We have 
\begin{align}
\frac{d}{dt} \vertiii{\mb{v}}^2_{h,D} =& - \|\mb{v}_t\|_{H\Lambda_{\alpha}}^2 - \|\mb{v}_t\|_{\tilde A^{(\beta^2)}}^2 -\begin{bmatrix}
     \mb{e}_{\mb{1}}^T \mb{v}_t\\
    \dl^T \mb{v}_t
\end{bmatrix}^T
\begin{bmatrix}
    \frac{\tau_1}{h} & \beta_1^2\\
    \beta_1^2 & h\theta\beta_{l,\min}^2
\end{bmatrix}
\begin{bmatrix}
     \mb{e}_{\mb{1}}^T \mb{v}_t\\
    \dl^T \mb{v}_t
\end{bmatrix}\notag \\ 
&-\begin{bmatrix}
     \mb{e}_{\mb{n}}^T \mb{v}_t\\
    \dr^T \mb{v}_t
\end{bmatrix}^T
\begin{bmatrix}
    \frac{\tau_2}{h} & \beta_n^2\\
    \beta_n^2 & h\theta\beta_{r,\min}^2
\end{bmatrix}
\begin{bmatrix}
     \mb{e}_{\mb{n}}^T \mb{v}_t\\
    \dr^T \mb{v}_t
\end{bmatrix}\leq 0,\label{EA1D_s6} 
\end{align}
where the discrete energy $\vertiii{\mb{v}}^2_{h,D}$ is 
\[
\frac{1}{2}\|\mb{v}_t\|_H^2+\frac{1}{2}\|\mb{v}\|_{A^{(\gamma^2)}}^2 +\gamma_1^2 \mb{v}^T \mb{e_1}\dl^T\mb{v} + \frac{\tau_3}{2h}\mb{v}^T \mb{e_1} \mb{e}_{\mb{1}}^T\mb{v} - \gamma_n^2 \mb{v}^T \mb{e_n}\dr^T\mb{v} + \frac{\tau_4}{2h}\mb{v}^T \mb{e_n} \mb{e}_{\mb{n}}^T\mb{v}.
\]
We note that the last four terms are approximately zero because of the homogeneous Dirichlet boundary condition. Thus, the discrete energy is an analogue of the continuous energy. 

We have shown that the right-hand side of \eqref{EA1D_s6} is nonpositive. Now we show that by choosing appropriate values of $\tau_3,\tau_4$, we have $\vertiii{\mb{v}}^2_{h,D}\geq 0$ so that it is indeed a discrete energy. We use \eqref{estimateA} to rewrite the term $\|\mb{v}\|_{A^{(\gamma^2)}}^2$, and obtain 
\begin{align*}
    \vertiii{\mb{v}}^2_{h,D}=&\frac{1}{2}\|\mb{v}_t\|_H^2+\frac{1}{2}\|\mb{v}\|_{\tilde A^{(\gamma^2)}}^2 +\frac{1}{2}h\bp \gamma^2_{l, \min} (\dl^T \mb{v})^2 + \frac{1}{2}h\bp \gamma^2_{r, \min} (\dr^T \mb{v})^2\\
    &+\gamma_1^2 \mb{v}^T \mb{e_1}\dl^T\mb{v} + \frac{\tau_3}{2h}\mb{v}^T \mb{e_1} \mb{e}_{\mb{1}}^T\mb{v} - \gamma_n^2 \mb{v}^T \mb{e_n}\dr^T\mb{v} + \frac{\tau_4}{2h}\mb{v}^T \mb{e_n} \mb{e}_{\mb{n}}^T\mb{v}\\
    =&\frac{1}{2}\|\mb{v}_t\|_H^2+\frac{1}{2}\|\mb{v}\|_{\tilde A^{(\gamma^2)}}^2 \\
&+\begin{bmatrix}
     \mb{e}_{\mb{1}}^T \mb{v}\\
    \dl^T \mb{v}
\end{bmatrix}^T
\begin{bmatrix}
    \frac{\tau_3}{h} & \gamma_1^2\\
    \gamma_1^2 & h\theta\gamma_{l,\min}^2
\end{bmatrix}
\begin{bmatrix}
     \mb{e}_{\mb{1}}^T \mb{v}\\
    \dl^T \mb{v}
\end{bmatrix}\notag\\
&+\begin{bmatrix}
     \mb{e}_{\mb{n}}^T \mb{v}\\
    \dr^T \mb{v}
\end{bmatrix}^T
\begin{bmatrix}
    \frac{\tau_4}{h} & \gamma_n^2\\
    \gamma_n^2 & h\theta\gamma_{r,\min}^2
\end{bmatrix}
\begin{bmatrix}
     \mb{e}_{\mb{n}}^T \mb{v}\\
    \dr^T \mb{v}
\end{bmatrix}.
\end{align*}
Therefore, we have $\vertiii{\mb{v}}^2_{h,D}\geq 0$ if
\begin{equation}\label{tau34}
 \tau_3\geq \frac{\gamma_1^4}{\theta\gamma_{l,\min}^2},\quad  \tau_4\geq \frac{\gamma_n^4}{\theta\gamma_{r,\min}^2}. 
\end{equation}

\rev{We note that the discrete energy is an approximation of the continuous energy, because the first term in $\vertiii{\mb{v}}^2_{h,D}$ approximates  $\|u_t\|^2_{L^2(\Omega)}$, and the remaining terms in $\vertiii{\mb{v}}^2_{h,D}$ approximate $\|\gamma \nabla u\|^2_{L^2(\Omega)}$. The dependence of  $\vertiii{\mb{v}}^2_{h,D}$ on the penalty parameters $\tau_3,\tau_4$ is an analogue of the discrete energy of the symmetric interior penalty discontinuous Galerkin method \cite{Grote2006}.} We have now obtained a discrete energy estimate, and summarize the result in the following theorem. 

\begin{theorem}\label{thm_dis_sta}
    The semidiscretization \eqref{semiD1D} satisfies the energy estimate \eqref{EA1D_s6} if the penalty parameters are chosen as in \eqref{tau1}, \eqref{tau2} and \eqref{tau34}.
\end{theorem}

\rev{
\begin{remark}
    The wave speed $\gamma$ is always positive, and the penalty parameters in \eqref{tau34} are thus valid. However, the viscous attenuation parameter $\beta$ can be zero, and the scheme needs to be adjusted accordingly. If $\beta=0$ on the left boundary, i.e., $\beta_1=0$, which leads to $\beta_{l,\min}=0$, then we only need to set the corresponding penalty parameter $\tau_1$ to zero. The energy estimate follows, because all terms associated with the left boundary vanish in the discretization. Similarly, we set $\tau_2=0$ when $\beta_n=0$.

    If $\beta$ is not equal to zero on the boundary, but is zero on some grid point near the boundary, then we discretize using the fully-compatible SBP operator $\hat D_{xx}^{(\beta^2)}$,
    \begin{equation}\label{semiD1Dfc}
    \mb v_{tt} + \Lambda_{\alpha}\mb{v}_t - \hat D_{xx}^{(\beta^2)}\mb{v}_t - D_{xx}^{(\gamma^2)}\mb{v} + \hat\FD = \mb{f} ,
\end{equation}
where
\begin{align}
\hat\FD =& -H^{-1} (-\beta_1^2 \mb{e_1}\hat{\mb{d}}_{\mb{1}}^{T} + \beta_n^2 \mb{e_n}\hat{\mb{d}}_{\mb{n}}^{T})^T \mb{v}_t +  H^{-1} \left(\frac{\hat\tau_1}{h}\mb{e}_1\mb{e}_1^T + \frac{\hat\tau_2}{h} \mb{e}_n\mb{e}_n^T \right)\mb{v}_t\notag \\
    & -H^{-1} (-\gamma_1^2 \mb{e_1}\dl^T + \gamma_n^2 \mb{e_n}\dr^T)^T \mb{v} + H^{-1} \left(\frac{\tau_3}{h}\mb{e}_1\mb{e}_1^T + \frac{\tau_4}{h} \mb{e}_n\mb{e}_n^T \right)\mb{v}. \label{FDfc}
\end{align}
Following the same procedure as the above analysis and \eqref{InverseInequality2}, the discretization \eqref{semiD1Dfc}-\eqref{FDfc} is stable with 
\[
\hat\tau_1\geq \frac{\beta_1^2}{\omega_1}, \quad \hat\tau_2\geq \frac{\beta_n^2}{\omega_1}, 
\]
and $\tau_3$ and $\tau_4$ from \eqref{tau34}. 
\end{remark}
}
    
\subsection{An SBP-SAT discretization for the Neumann problem}
We consider the governing equation in domain $\Omega=[0,1]$ with the homogeneous Neumann boundary condition \eqref{NeumannBC}, 
\begin{equation}\label{semiN1D}
    \mb v_{tt} + \Lambda_{\alpha}\mb{v}_t - D_{xx}^{(\beta^2)}\mb{v}_t - D_{xx}^{(\gamma^2)}\mb{v} + \FN = \mb{f},
\end{equation}
where $\FN$ imposes weakly the Neumann boundary condition and is determined through stability analysis. The other terms in \eqref{semiN1D} are the same as in \eqref{semiD1D}.

To prove stability, we consider zero forcing, and multiply equation \eqref{semiN1D} by $\mb{v}_t^T H$. By using the SBP identity \eqref{SBPidentityD2} and  
rearranging terms, we obtain 
\begin{align}
&\frac{d}{dt}\left(\frac{1}{2}\|\mb{v}_t\|_H^2+\frac{1}{2}\|\mb{v}\|_{A^{(\gamma^2)}}^2\right) \label{DiscreteEnergy1D_N}\\
=& - \|\mb{v}_t\|_{H\Lambda_{\alpha}}^2 - \|\mb{v}_t\|_{A^{(\beta^2)}}^2 + \mb{v}_t^T \left(-\beta_1^2 \mb{e_1}\dl^T + \beta_n^2 \mb{e_n}\dr^T \right)\mb{v}_t\notag\\ 
&+ \mb{v}_t^T \left(-\gamma_1^2 \mb{e_1}\dl^T + \gamma_n^2 \mb{e_n}\dr^T \right)\mb{v}  -  \mb{v}_t^T H \FN.\notag
\end{align}
We need to design $\FN$ such that $\FN$ is consistent with the boundary condition $u_x=0$, and we can obtain a discrete energy estimate from \eqref{DiscreteEnergy1D_N}. Since all boundary terms in \eqref{DiscreteEnergy1D_N} are the approximations of the derivative at the boundary, it is straightforward to use $\FN$ to cancel all boundary terms. To this end, we have
\begin{align}
\FN = H^{-1} (-\beta_1^2 \mb{e_1}\dl^T + \beta_n^2 \mb{e_n}\dr^T) \mb{v}_t + H^{-1} (-\gamma_1^2 \mb{e_1}\dl^T + \gamma_n^2 \mb{e_n}\dr^T) \mb{v}. \label{FN}
\end{align}
Unlike the Dirichlet boundary condition, there is no penalty parameter in \eqref{FN}. Substituting the expression \eqref{FN} to \eqref{DiscreteEnergy1D_N}, we obtain the following discrete energy estimate
\begin{align}\label{DiscreteEnergy1D_N2}
\frac{d}{dt}\left(\frac{1}{2}\|\mb{v}_t\|_H^2+\frac{1}{2}\|\mb{v}\|_{A^{(\gamma^2)}}^2\right) =  - \|\mb{v}_t\|_{H\Lambda_{\alpha}}^2 - \|\mb{v}_t\|_{A^{(\beta^2)}}^2 \leq 0.
\end{align}
We summarize the stability result in the following theorem. 
\begin{theorem}\label{thm_dis_sta_N}
    The semidiscretization \eqref{semiN1D} with \eqref{FN} satisfies the energy estimate \eqref{DiscreteEnergy1D_N2}.
\end{theorem}

\subsection{Multidimensional problems}
We consider a rectangular shaped domain in two space dimension, $\Omega=[0,1]^2$. The one dimensional SBP operators defined in Sec. \ref{sec_SBP_1D} can be generalized to two dimension by using Kronecker product. To see this, we discretize $\Omega=[0,1]^2$ by $n$ grid points in each spatial direction, and use a column-wise ordering. For example, we denote the pointwise evaluation of a function $u(x,y)$ on the grid point $(x_i, y_j)$ as $u_{ij}$, and store in a vector 
\[
\tilde{\mb{u}} = [u_{11}, u_{12},\cdots, u_{1n}, u_{21}, u_{22},\cdots, u_{2n}, \cdots, u_{n1}, u_{n2},\cdots, u_{nn}]^T.
\]
We use the tilde-symbol for variables in two dimension. 

Let $\tilde D_{xx}^{(b)}$ be the SBP operator approximating the second derivative in the $x$ direction, $\frac{\partial}{\partial x} b(x,y)\frac{\partial}{\partial x}$, then $\tilde D_{xx}^{(b)}$ can be constructed by repeatedly using the one dimensional SBP operator on every horizontal grid line,
\begin{equation*}
    \tilde D_{xx}^{(b)} = \sum_{j=1}^n D_2^{(b(\mathbf x,y_j))}\otimes E_j,
\end{equation*}
where $D_2^{(b(\mathbf x,y_j))}$ is the one dimensional SBP operator defined on grid line $(\mathbf x,y_j)$, and $\mathbf x=[x_1, x_2,\cdots,x_n]^T$. The $n$-by-$n$ matrix $E_j$ has components zero everywhere except in column $j$ and row $j$, where the component equals to one. Similarly, the SBP operator $\tilde D_{yy}^{(b)}$ for  $\frac{\partial}{\partial y} b(x,y)\frac{\partial}{\partial y}$ can be constructed as 
\begin{equation*}
    \tilde D_{yy}^{(b)} = \sum_{j=1}^n E_j\otimes D_2^{(b(x_j, \mathbf y))},
\end{equation*}
by using the one dimensional SBP operator on every vertical grid line. We also define the operator
\begin{equation}\label{D_Delta}
    \tilde D_{\Delta}^{(b)} = \tilde D_{xx}^{(b)} + \tilde D_{yy}^{(b)},
\end{equation}
which approximates the Laplacian $\nabla\cdot b\nabla$.

The semidiscretization for the governing equation \eqref{DVWE} with homogeneous Dirichlet boundary condition \eqref{DirichletBC} \rev{and $\beta>0$} is 
\begin{equation}\label{semiD2D}
    \tilde{\mb v}_{tt} + \tilde \Lambda_{\alpha} \tilde{\mb{v}}_t - \tilde D_{\Delta}^{(\beta^2)} \tilde{\mb{v}}_t - \tilde D_{\Delta}^{(\gamma^2)} \tilde{\mb{v}} + \tFD = \tilde{\mb{f}} ,
\end{equation}
where the vector $\tilde{\mb v}$ is the finite difference solution, $\tilde \Lambda_{\alpha}$ is a diagonal matrix with the pointwise evaluation of the diffusive attenuation parameter $\alpha(x,y)$ on the grid, i.e., $(\tilde \Lambda_{\alpha})_{kk} = \alpha(x_i, y_j)$, where $k=(i-1)n+j$. The vector $\tilde{\mb{f}}$ is pointwise evaluation of the forcing function on the grid. The SBP operators $\tilde D_{\Delta}^{(\beta^2)}$ and $\tilde D_{\Delta}^{(\gamma^2)}$ are constructed according to \eqref{D_Delta}. The term $\tFD$ imposes weakly the boundary condition, and can also be constructed by using its one dimensional counterpart \eqref{FD}. More precisely, we have 
\[
\tFD = \tFDx + \tFDy,
\]
where $\tFDx$ and $\tFDy$ impose boundary conditions in the $x$-direction and $y$-direction, respectively. They take the form
\begin{align}
\tFDx =& -\tilde H^{-1}_x \left(-\sum_{j=1}^n \beta_{1j}^2 \mb{e_1}\dl^T \otimes E_j + \sum_{j=1}^n \beta_{nj}^2 \mb{e}_n\mb{d}_n^T \otimes E_j \right)^T \tilde{\mb{v}}_t \notag \\
&+ \tilde H^{-1}_x \left(\frac{1}{h} \sum_{j=1}^n \tau_{1j} \mb{e}_1\mb{e}_1^T\otimes E_j + \frac{1}{h} \sum_{j=1}^n \tau_{2j} \mb{e}_n\mb{e}_n^T\otimes E_j \right)\tilde{\mb{v}}_t \notag  \\
& -\tilde H^{-1}_x \left(-\sum_{j=1}^n \gamma_{1j}^2 \mb{e_1}\dl^T \otimes E_j + \sum_{j=1}^n \gamma_{nj}^2 \mb{e}_n\mb{d}_n^T \otimes E_j \right)^T \tilde{\mb{v}} \notag \\
&+ \tilde H^{-1}_x \left(\frac{1}{h} \sum_{j=1}^n \tau_{3j} \mb{e}_1\mb{e}_1^T\otimes E_j + \frac{1}{h} \sum_{j=1}^n \tau_{4j} \mb{e}_n\mb{e}_n^T\otimes E_j \right)\tilde{\mb{v}}, \notag
\end{align}

\begin{align}
\tFDy =& -\tilde H^{-1}_y \left(-\sum_{j=1}^n  E_j \otimes \beta_{j1}^2 \mb{e_1}\dl^T + \sum_{j=1}^n E_j \otimes \beta_{jn}^2    \mb{e}_n\mb{d}_n^T\right)^T \tilde{\mb{v}}_t \notag \\
&+ \tilde H^{-1}_y \left(\frac{1}{h} \sum_{j=1}^n E_j \otimes   \hat\tau_{1j} \mb{e}_1\mb{e}_1^T + \frac{1}{h} \sum_{j=1}^n E_j\otimes \hat\tau_{2j} \mb{e}_n\mb{e}_n^T \right)\tilde{\mb{v}}_t \notag  \\
& -\tilde H^{-1}_y \left(-\sum_{j=1}^n E_j \otimes  \gamma_{j1}^2 \mb{e_1}\dl^T  + \sum_{j=1}^n E_j \otimes \gamma_{jn}^2 \mb{e}_n\mb{d}_n^T  \right)^T \tilde{\mb{v}} \notag \\
&+ \tilde H^{-1}_y \left(\frac{1}{h} \sum_{j=1}^n E_j \otimes \hat\tau_{3j} \mb{e}_1\mb{e}_1^T + \frac{1}{h} \sum_{j=1}^n E_j \otimes  \hat\tau_{4j} \mb{e}_n\mb{e}_n^T \right)\tilde{\mb{v}}. \notag
\end{align}
In the above, the material parameters are defined as $\beta_{ij}=\beta(x_i,y_j)$ and 
$\gamma_{ij}=\gamma(x_i,y_j)$. We have also used the notation $\tilde H_x = H_x \otimes I_y$ and $\tilde H_y = I_x \otimes H_y$, where $H_x,H_y$ are the SBP norm matrices and $I_x,I_y$ are identity matrices. The penalty parameters $\tau_{1j}$, $\tau_{2j}$, $\tau_{3j}$, $\tau_{4j}$, $\hat\tau_{1j}$, $\hat\tau_{2j}$, $\hat\tau_{3j}$, $\hat\tau_{4j}$ are determined through stability analysis. Since the SBP operators \eqref{D_Delta} are defined in a dimension-by-dimension manner, the stability analysis essentially follows from the corresponding one dimensional problem. To this end, we omit the proof and only state the choices of the penalty parameters so that an energy estimate can be obtained. 

\begin{theorem}
    The semidiscretization \eqref{semiD2D} satisfies an energy estimate if the penalty parameters satisfy 
    \begin{align*}
        \tau_{1j} \geq \frac{\beta_{1j}^4}{\theta\beta_{l,j,\min}^2},\ \tau_{2j} \geq \frac{\beta_{nj}^4}{\theta\beta_{r,j,\min}^2}, \  \tau_{3j} \geq \frac{\gamma_{1j}^4}{\theta\gamma_{l,j,\min}^2},\ \tau_{4j} \geq \frac{\gamma_{nj}^4}{\theta\gamma_{r,j,\min}^2}, \\
        \hat\tau_{1j} \geq \frac{\beta_{j1}^4}{\theta\beta_{b,j,\min}^2},\ \hat\tau_{2j} \geq \frac{\beta_{jn}^4}{\theta\beta_{t,j,\min}^2}, \  \hat\tau_{3j} \geq \frac{\gamma_{j1}^4}{\theta\gamma_{b,j,\min}^2},\ \textbf{}\tau_{4j} \geq \frac{\gamma_{jn}^4}{\theta\gamma_{t,j,\min}^2}.
    \end{align*}
\end{theorem}

\rev{
\begin{remark}
    If $\beta$ is equal to zero on a grid point on the boundary, then we set the corresponding penalty parameter equal to zero. If $\beta$ is not equal to zero on the boundary, but is zero on some grid point
near the boundary, then we approximate the viscous term by using the fully-compatible SBP operator. The stability estimate follows in the same way as for the 1D problem, see Remark 1. 
\end{remark}
}

The semidiscretization for the Neumann problem takes the same form as in \eqref{semiD2D} by replacing $\tFD$ by $\tFN$,
\begin{align}
\tFN =& \tilde H^{-1}_x \left(-\sum_{j=1}^n \beta_{1j}^2 \mb{e_1}\dl^T \otimes E_j + \sum_{j=1}^n \beta_{nj}^2 \mb{e}_n\mb{d}_n^T \otimes E_j \right) \tilde{\mb{v}}_t \notag \\
& \tilde H^{-1}_x \left(-\sum_{j=1}^n \gamma_{1j}^2 \mb{e_1}\dl^T \otimes E_j + \sum_{j=1}^n \gamma_{nj}^2 \mb{e}_n\mb{d}_n^T \otimes E_j \right) \tilde{\mb{v}} \notag \\
& \tilde H^{-1}_y \left(-\sum_{j=1}^n  E_j \otimes \beta_{j1}^2 \mb{e_1}\dl^T + \sum_{j=1}^n E_j \otimes \beta_{jn}^2    \mb{e}_n\mb{d}_n^T\right) \tilde{\mb{v}}_t \notag \\
& \tilde H^{-1}_y \left(-\sum_{j=1}^n E_j \otimes  \gamma_{j1}^2 \mb{e_1}\dl^T  + \sum_{j=1}^n E_j \otimes \gamma_{jn}^2 \mb{e}_n\mb{d}_n  \right) \tilde{\mb{v}}. \notag 
\end{align}
Similarly, the semidiscretization also satisfies an energy estimate.

\section{Error estimates}\label{sec-err}
The truncation error of the SBP operator $D_2^{(b)}$ is $\mathcal{O}(h^{2p})$ in the interior, and $\mathcal{O}(h^{p})$ on the first $k$ grid points near each boundary. In the case of constant coefficient, we have $k=p$ for $p=1,2,3$. In the spatial discretization, it is often the boundary truncation error  $\mathcal{O}(h^{p})$ that determines the convergence rate. The precise convergence rate depends on the equation, the boundary condition, and how boundary condition is imposed. Because the number of grid points  with the truncation error $\mathcal{O}(h^{p})$ is independent of the mesh size, the overall convergence rate can be expected to be higher than $p$. The energy estimate of the error equation predicts a convergence rate $p + 1/2$, but this is suboptimal and higher rates are observed in numerical experiments. Sharper error estimates can be derived using the normal mode analysis, which gives a convergence rate of $p + 2$ for many problems, though there are special cases with lower or higher rates \cite{Wang2017,Wang2023,Wang2018b}. 

In this section, we derive an a priori error estimate for the semidiscretization \eqref{semiD1D} and consider the case with constant coefficients. The discretization in a bounded domain can be split into three parts and analyzed separately, consisting of two half-line problems in $[0,\infty)$ and $(-\infty,1]$, and one Cauchy problem in the unbounded domain $(-\infty,\infty)$. For the unbounded problem, central finite difference stencils with truncation error $\mathcal{O}(h^{2p})$ are used on all grid points, resulting in a convergence rate $2p$. Thus, we only need to analyze the two half-line problems. Without loss of generality, we consider the half-line problem on $[0,\infty)$, as the other half-line problem can be analyzed in the same way.

Let $\mb{u}=[u_1,u_2,\cdots]^T$ be the exact solution evaluated on the grid, i.e., $u_j = u(x_j, t)$ for $x_j = (j-1)h$ and $j=1,2,\cdots$, where $h$ is the grid size. We define the pointwise error $\bs{\varepsilon}=\mb{u}-\mb{v}$ with components $\pe_j = u_j - v_j$, which satisfies the error equation 
\begin{equation}\label{errD1D}
     \bs{\pe}_{tt} + \alpha\bs{\pe}_t - \beta^2 D_{xx}\bs{\pe}_t - \gamma^2 D_{xx} \bs{\pe} + \FDe = \mb{T}, 
\end{equation}
where $\FDe$ corresponds to the terms that impose the Dirichlet boundary condition at $x=0$, 
\begin{align}
    \FDe =& -H^{-1} (-\beta_1^2 \mb{e_1}\dl^T)^T \bs{\pe}_t +  H^{-1} \left(\frac{\tau_1}{h}\mb{e}_1\mb{e}_1^T  \right)\bs{\pe}_t\notag \notag\\
    & -H^{-1} (-\gamma_1^2 \mb{e_1}\dl^T )^T \bs{\pe} + H^{-1} \left(\frac{\tau_3}{h}\mb{e}_1\mb{e}_1^T\right)\bs{\pe}.\label{FDe}
\end{align}
The operator $D_{xx}$ approximates the second derivative $d^2/dx^2$, and is the same as $D_{xx}^{(b)}$ when $b(x)\equiv 1$. The right-hand side $\mb{T}$ is the truncation error and takes the form 
\begin{align*}
\mb{T} =& [\underbrace{\mathcal{O}(h^p),\cdots,\mathcal{O}(h^p)}_{\text{the first $p$ grid points}}, \mathcal{O}(h^{2p}),\cdots\cdots,\mathcal{O}(h^{2p}),\cdots\cdots]^T .    \\
=&\underbrace{[\mathcal{O}(h^p),\cdots,\mathcal{O}(h^p), 0,\cdots\cdots,0,\cdots\cdots]^T}_{\mb{T}_p} \\
&+\underbrace{[0,\cdots,0, \mathcal{O}(h^{2p}),\cdots\cdots,\mathcal{O}(h^{2p}),\cdots\cdots]^T}_{\mb{T}_{2p}}.
\end{align*}
Here, $\mb{T}$ is split into two parts, the boundary truncation error $\mb{T}_p$ and the interior truncation error $\mb{T}_{2p}$. There are only $p$ nonzero components in $\mb{T}_p$, while there are only $p$ zeros in $\mb{T}_{2p}$. As a consequence, we have $\|\mb{T}_p\|_H=\mathcal{O}(h^{p+1/2})$ and $\|\mb{T}_{2p}\|_H=\mathcal{O}(h^{2p})$.

Accordingly, we also split the pointwise error as $\bs{\pe}=\bs{\pep}+\bs{\pepp}$ such that 
\begin{align}
     \bs{\pep}_{tt} + \alpha\bs{\pep}_t - \beta^2 D_{xx}\bs{\pep}_t - \gamma^2 D_{xx} \bs{\pep} + \FDp = \mb{T}_p, \label{pep_eqn}\\
     \bs{\pepp}_{tt} + \alpha\bs{\pepp}_t - \beta^2 D_{xx}\bs{\pepp}_t - \gamma^2 D_{xx} \bs{\pepp} + \FDpp = \mb{T}_{2p}, \label{pepp_eqn}
\end{align}
where $\FDp$ and $\FDpp$ are defined in the same way as $\FDe$ in \eqref{FDe} by replacing $\bs\pe$ by $\bs\pep$ and $\bs\pepp$, respectively. The error component $\bs\pepp$, driven by the interior truncation error $\mb{T}_{2p}$, can be estimated by applying the energy method to \eqref{pepp_eqn}, leading to $\vertiii{\bs\pepp}_{h,D,0}\leq C\|\bs{T}_{2p}\|_H\leq Ch^{2p}$. The subscript $0$ in the discrete energy norm indicates that only boundary contribution from the left boundary $x=0$ is included. 

It is the error component $\bs\pep$ that dominates the pointwise error $\bs\pe$. To estimate $\bs\pep$ by the normal mode analysis, we consider a particular choice $p=2$, i.e., the SBP operator $D_{xx}$ consists of the fourth order accurate central finite difference stencil in the interior on grid points $x_j, j=5,6,\cdots$, and second order accurate boundary closure on the first four grid points $x_j,j=1,2,3,4$. 

To continue, we Laplace transform \eqref{pep_eqn} and obtain 
\begin{align}
     s^2\hat{\bs{\pep}} + \alpha s\hat{\bs{\pep}} - \beta^2 s D_{xx}\hat{\bs{\pep}} - \gamma^2 D_{xx} \hat{\bs{\pep}} + \FDpL = {\hat{\mb{T}}}_2, \label{pep_eqnL}
\end{align}
where the hat-symbol denotes variables in the Laplace space, and $s$ is the time dual. On the interior grid points, \eqref{pep_eqnL} reduces to 
\begin{equation}\label{pep_eqnLi}
     s^2\hat{\pep}_j + \alpha s\hat{\pep}_j - (\beta^2 s + \gamma^2) D_{xx,I}\hat{\pep}_j  = 0,\quad j=5,6,\cdots,
\end{equation}
where the fourth order accurate finite difference stencil is 
\[
 D_{xx,I}\hat{\pep}_j = \frac{1}{h^2}\left(-\frac{1}{12}\hat{\pep}_{j-2} + \frac{4}{3}\hat{\pep}_{j-1} - \frac{5}{2}\hat{\pep}_{j}  + \frac{4}{3}\hat{\pep}_{j+1} -\frac{1}{12}\hat{\pep}_{j+2}  \right). 
\]
We note that in the interior, the boundary term  $\FDpL$ has no contribution, and the truncation error ${\hat{\mb{T}}}_p$ only contains zeros. 

Rearranging the terms in \eqref{pep_eqnLi}, we have the following difference equation, 
\begin{equation}\label{interior_diff}
-\frac{1}{12}\hat{\pep}_{j-2}+\frac{4}{3}\hat{\pep}_{j-1}-\left(\frac{5}{2}+h^2\frac{s^2+\alpha s}{\beta^2 s + \gamma^2}\right)\hat{\pep}_{j} + \frac{4}{3}\hat{\pep}_{j+1}-\frac{1}{12}\hat{\pep}_{j+2}=0,\quad j=5,6,\cdots. 
\end{equation}
The corresponding characteristic equation is 
\begin{equation}\label{cha_eqn}
-\frac{1}{12}+\frac{4}{3}\kappa - \left(\frac{5}{2}+\underbrace{h^2\frac{s^2+\alpha s}{\beta^2 s + \gamma^2}}_{r}\right)\kappa^2+\frac{4}{3}\kappa^3-\frac{1}{12}\kappa^4 = 0,
\end{equation}
and has four roots. Since the semidiscretization satisfies an energy estimate and is thus stable, only roots with $|\kappa|<1$ are admissible, and it suffices to consider bounded $s$ with $Re(s)>0$ and $|s|<C$ for some constant $C$ independent of $h$. Consequently, in the asymptotic regime when $h$ goes to zero, we have $|r|$ goes to zero. Solving \eqref{cha_eqn}, we find that the admissible roots satisfying $|\kappa|<1$ are 
\begin{equation}\label{AdmissibleRoots}
\kappa_1 = q + 4 - \sqrt{8q-3r+24},\quad \kappa_2 = -q + 4 - \sqrt{-8q-3r+24},
\end{equation}
where $q=\sqrt{9-3r}$ and $r$ is defined in \eqref{cha_eqn}. When $s=0$, we have $|\kappa_1|=7-4\sqrt{3}\approx 0.0718<1$ and $|\kappa_2|=1$. Thus, $\kappa_1$ corresponds to a fast decaying mode and $\kappa_2$ corresponds to a slowly decaying mode. 

On the first four grid points, the SBP operator $D_{xx}$ has one-sided boundary closure and $\hat{\mb{T}}_2$ has nonzero components. Equation \eqref{pep_eqnL} becomes
\begin{equation}\label{pep_eqnLB}
 s^2\hat{\bs{\pep}}_B + \alpha s\hat{\bs{\pep}}_B - \beta^2 s D_{xx,B}\hat{\bs{\pep}}_B - \gamma^2 D_{xx,B} \hat{\bs{\pep}}_B + \FDpLB = {\hat{\mb{T}}}_{2,B},    
\end{equation}
where $\hat{\bs{\pep}}_B$, $\FDpLB$ and ${\hat{\mb{T}}}_{2,B}$ are 4-by-1 vectors containing the first four components of $\hat{\bs{\pep}}$, $\FDpL$ and ${\hat{\mb{T}}}_{2}$, respectively. The boundary closure is stored in the 4-by-6 matrix $D_{xx,B}$, and the precise components can be found in \cite{Mattsson2004}. 

Equation \eqref{interior_diff} is a linear recurrence relation with two admissible roots $\kappa_1$ and $\kappa_2$. Using them,  we make the ansatz
\[
\hat{{\pep}}_j = \sigma_1 \kappa_1^{j-3} +  \sigma_2 \kappa_2^{j-3},\quad j = 3,4,\cdots,
\]
where the coefficients $\sigma_1,\sigma_2$ and the first two components $\hat{{\pep}}_1,\hat{{\pep}}_2$ are determined by using the boundary closure in \eqref{pep_eqnLB}. Then, the error $\hat{\bs{\pep}}$ in $l^2$ norm is 
\begin{align}
\|\hat{{\bs\pep}}\|_h^2 =& h\sum_{j=1}^{\infty}|\hat{{\pep}}_j|^2 \notag \\
=& h (|\hat{{\pep}}_1|^2 + |\hat{{\pep}}_2|^2) + \frac{h|\sigma_1|^2}{1-|\kappa_1|^2} + \frac{h|\sigma_2|^2}{1-|\kappa_2|^2}.\label{estimate_pep_1}
\end{align}
We need the following lemma for $\kappa_2$ to bound \eqref{estimate_pep_1}.

\begin{lemma}
    The admissible root $\kappa_2$ in \eqref{AdmissibleRoots} satisfies 
    \begin{equation}\label{estimate_kappa2}
        \frac{1}{1-|\kappa_2|^2} \leq Ch^{-1}, 
    \end{equation}
    for a constant $C$ that depends on the material parameters but not $h$. 
\end{lemma}

\begin{proof}
Using the formula in  \eqref{AdmissibleRoots}, we expand $\kappa_2$ for small $h$, 
\begin{equation}
    \kappa_2 = 1 - \frac{\sqrt{s^2+\alpha s}}{\sqrt{s\beta^2+\gamma^2}}h + \frac{s^2+\alpha s}{2(s\beta^2+\gamma^2)}h^2 + \mathcal{O}(h^3).
\end{equation}
    We consider $Re(s)=\eta>0$ for some $\eta$, and 
    denote the coefficient for $h$ as $\frac{\sqrt{s^2+\alpha s}}{\sqrt{s\beta^2+\gamma^2}}=c_r+ic_i$ for real $c_r, c_i$ that also depend on $\eta$. We obtain to the leading order,
    \begin{align*}
        \frac{1}{1-|\kappa_2|^2} &\leq  \frac{c}{1-|1-c_r h-i c_i h|^2} \leq \frac{c}{2c_r h}.
    \end{align*}
    The desired estimate \eqref{estimate_kappa2} follows by setting $C=c/(2c_r)$. 
\end{proof}
We note that the slow-decaying component $\kappa_2$ satisfies $|\kappa_2|=1$ when $s=0$, and $|\kappa_2|=1+\mathcal{O}(h)$ in the vicinity of $s=0$, which leads to $1/(1-|\kappa_2|^2)=\mathcal{O}(h^{-1})$. The corresponding term for the fast-decaying component $\kappa_1$ can be bounded independent of $h$, i.e., $1/(1-|\kappa_1|^2)=\mathcal{O}(1)$. In the following, we estimate the four unknown components $\sigma_1,\sigma_2,\hat{{\pep}}_1,\hat{{\pep}}_2$, which will be combined with \eqref{estimate_kappa2} to derive an estimate for \eqref{estimate_pep_1}.

Equivalently, we can write  \eqref{pep_eqnLB} as a linear system 
\begin{equation}\label{BS}
A\Sigma = h^2{\hat{\mb{T}}}_{2,B},
\end{equation}
where $\Sigma=[\hat{{\pep}}_1,\hat{{\pep}}_2,\sigma_1,\sigma_2]^T$, and $h^2$ on the right-hand side comes from the $h^2$ factor in $D_{xx,B}$. Since the boundary closure is second order accurate, we have $h^2{\hat{\mb{T}}}_{2,B}=\mathcal{O}(h^{4})$. The 4-by-4 matrix $A$ depends on $s$, $h$, the material parameters $\alpha,\beta,\gamma$, and the penalty parameters. In a stable semidiscretization, $A$ is nonsingular for all $s$ with $Re(s)>0$ \cite{Gustafsson2013}. Thus, we may write the solution to \eqref{BS} as $\Sigma = h^2 A^{-1} {\hat{\mb{T}}}_{2,B}$. To derive an estimate for $\Sigma$, it is important to analyze the dependence of $A^{-1}$ on $h$. 

If $A$ is nonsingular when $s=0$, then the so-called determinant condition is satisfied \cite{Gustafsson2013}. In this case, all components of $\Sigma$ are $\mathcal{O}(h^{4})$. Consequently, we have $\|\hat{{\bs\pep}}\|_h = \mathcal{O}(h^{4})$ in \eqref{estimate_pep_1}.

If the determinant condition is not satisfied, i.e., $A$ is singular when $s=0$, then a perturbation analysis is needed to obtain the precise dependence of $\Sigma$ on $h$.  It is important to note that components of $\Sigma$ may depend on $h$ in different ways \cite{Wang2017}. In fact, to obtain  $\|\hat{{\bs\pep}}\|_h = \mathcal{O}(h^{4})$, it is enough to have   $\sigma_1,\sigma_2,\hat{{\pep}}_1= \mathcal{O}(h^{3})$, i.e., we can afford to lose one order in these variables because of a singular $A$. The coefficient $\sigma_2$ multiplying with the term of the slow-decaying component $\kappa_2$ must be 
$\mathcal{O}(h^{4})$. Since the linear system \eqref{BS} depends on the material parameters, we divide the analysis into four cases. 

\paragraph{Case 1: $\alpha=\beta=0$} In this case, the diffusive viscous wave equation \eqref{DVWE} reduces to the wave equation. The corresponding a priori error estimates were derived in \cite{Wang2017}, and we refine the analysis below.

When $\beta=0$, the terms with $\bs{\pep}_t$ in $\FDpLB$ in \eqref{pep_eqnLB} vanishes, and only one penalty parameter $\tau_3$ remains. By Theorem \ref{thm_dis_sta}, the discretization is stable if $\tau_3\geq \gamma^2/\theta$. The determinant condition is satisfied for all $\tau_3 > \gamma^2/\theta$, and $\|\hat{{\bs\pep}}\|_h = \mathcal{O}(h^{4})$ follows. On the stability limit $\tau_3 = \gamma^2/\theta$, the determinant condition is not satisfied, in which case the error estimate is obtained by the energy estimate in \cite{Wang2017}. Below we show that it can be analyzed directly by solving \eqref{BS}. 

Consider $s\neq 0$ in a vicinity of the origin, then $A$ is nonsingular. For small $h$, we solve the boundary system \eqref{BS}, and obtain
\begin{equation}\label{case1}
\hat{{\pep}}_1,\hat{{\pep}}_2,\sigma_1= \mathcal{O}(s^{-2}h^{2}), \quad \sigma_2=\mathcal{O}(h^{4}). 
\end{equation}
This means that the singularity of $A$ at $s=0$ does not affect the coefficient for the slow-decaying component $\kappa_2$, but two orders in $h$ are lost in the other three variables $\hat{{\pep}}_1, \hat{{\pep}}_2, \sigma_1$. Substituting \eqref{case1} into \eqref{estimate_pep_1}, we obtain 
\begin{equation}\label{case1_2.5}
    \|\hat{{\bs\pep}}\|_h \leq Ch^{2.5}.
\end{equation}

The above estimate \eqref{case1_2.5} gives a convergence rate 2.5. In addition, the error is dominated by the two pointwise errors $\hat{{\pep}}_1, \hat{{\pep}}_2$, and decays exponentially fast away from the boundary. Both the observed  convergence rate and the error behavior in numerical examples in \cite{Wang2017} agree with the above analysis. 

\paragraph{Case 2: $\alpha\neq 0, \beta=0$} 
The diffusive attenuation term is nonzero in the governing equation, but it does not affect the numerical boundary treatment.  By Theorem \ref{thm_dis_sta}, the discretization is stable if $\tau_3\geq \gamma^2/\theta$. Similar as in case 1, the determinant condition is satisfied if  $\tau_3 > \gamma^2/\theta$, which leads to $\|\hat{{\bs\pep}}\|_h = \mathcal{O}(h^{4})$. When  $\tau_3 = \gamma^2/\theta$, the matrix $A$ is singular at $s=0$. We solve \eqref{BS} in a vicinity of $s=0$, and obtain
\begin{equation}\label{case2}
\hat{{\pep}}_1,\hat{{\pep}}_2,\sigma_1= \mathcal{O}(s^{-1}h^{2}), \quad \sigma_2=\mathcal{O}(h^{4}). 
\end{equation}
Comparing \eqref{case2} with \eqref{case1}, we observe that all four variables have the same $h$-dependence. Thus, the same error estimate \eqref{case1_2.5} is obtained, and the convergence rate is 2.5. 

\paragraph{Case 3: $\alpha= 0, \beta\neq 0$}
The viscous attenuation term is nonzero, and play an important role in the error estimate. By Theorem \ref{thm_dis_sta}, the discretization is stable if $\tau_1\geq \beta^2/\theta$ and $\tau_3\geq \gamma^2/\theta$. We find that the determinant condition is not satisfied if  $\tau_3 = \gamma^2/\theta$ and  $\tau_1 =  \beta^2/\theta $. In this case, we solve \eqref{BS} in a vicinity of $s=0$, and obtain 
\begin{equation}\label{case3}
\hat{{\pep}}_1,\hat{{\pep}}_2,\sigma_1= \mathcal{O}(s^{-2}h^2), \quad \sigma_2=\mathcal{O}(h^{4}), 
\end{equation}
which is the same as in Case 1.

\paragraph{Case 4: $\alpha\neq 0, \beta\neq 0$} Both the diffusive and the viscous attenuation terms are nonzero. The situation is similar as in Case 3, that the determinant condition is not satisfied if  $\tau_3 = \gamma^2/\theta$ and  $\tau_1 =  \beta^2/\theta $. The solution to \eqref{BS} in a vicinity of $s=0$ is 
\begin{equation}\label{case4}
\hat{{\pep}}_1,\hat{{\pep}}_2,\sigma_1= \mathcal{O}(s^{-1}h^{2}), \quad \sigma_2=\mathcal{O}(h^{4}),
\end{equation}
which is the same as in Case 2.

In the above, we have derived error estimates for $\|\hat{{\bs\pep}}\|_h$ in Laplace space. Parseval's relation can be used to obtain the corresponding error estimates in physical space to the same order in $h$, see \cite{Wang2017}. 
In conclusion,  it is important to choose the penalty parameters strictly larger than the value required by stability, so that the fourth order convergence rate is obtained.  

\rev{
\begin{remark}
The normal mode analysis can only be carried out for the constant coefficient case \cite{Gustafsson2013}. For problems with variable coefficients, a priori error estimates can be derived by the energy estimate. Consider the semidiscretrization \eqref{semiD1D}, the pointwise error $\bs{\varepsilon}=\mb{u}-\mb{v}$ satisfies the error equation 
\begin{equation}\label{errD1Dvc}
     \bs{\pe}_{tt} + \Lambda_{\alpha}\bs{\pe}_t - D_{xx}^{(\beta^2)}\bs{\pe}_t -  D_{xx}^{(\gamma^2)} \bs{\pe} + \FDe = \mb{T}, 
\end{equation}
which has exactly the same form as  \eqref{semiD1D} with the forcing function $\mb{f}$ replaced by the truncation error $\mb{T}$. Since  \eqref{semiD1D} satisfies an energy estimate, the error equation \eqref{errD1Dvc} also satisfies an energy estimate, 
\[
 \vertiii{\bs{\varepsilon}}_{h,D}\leq C\|\mb{T}\|_h. 
\]

With the standard SBP operators from \cite{Mattsson2012}, the truncation error takes the form 
\[
\mb{T}= [\underbrace{\mathcal{O}(h^p),\cdots,\mathcal{O}(h^p)}_{\text{the first $k$ grid points}}, \mathcal{O}(h^{2p}),\cdots\cdots,\mathcal{O}(h^{2p}),\underbrace{\mathcal{O}(h^p),\cdots,\mathcal{O}(h^p)}_{\text{the last $k$ grid points}}]^T,
\]
and we have $\|\mb{T}\|_h=\mathcal{O}(h^{p+1/2})$ because the number of grid points with truncation error $\mathcal{O}(h^p)$ is independent of $h$. In the case when $p=2$, this gives a convergence rate of 2.5 in the energy norm, which amounts to 3.5 in $l^2$ norm. 
\end{remark}
}

\section{Numerical experiments}
We present numerical examples to verify stability and accuracy properties of the developed method. In all examples, we examine properties related to the spatial discretization. For time integration, we choose the classical fourth order accurate Runge-Kutta method and use a stepsize small enough so that the error in the numerical solution is dominated by the spatial discretization. For the case when  the viscous attenuation term is nonzero, explicit time integration requires a parabolic-type restriction on the time step, chosen to be $0.1h^2$.   This can be improved by using implicit methods with stepsize restriction $\sim h$. 

At the final time, we compute the error in the discrete $l^2$ norm as 
\[
\|\mb u - \mb v \|_{l^2} = \sqrt{h^d \sum_{j=1}^n (u_j-v_j)^2 },
\]
where the vector $\mb u$ contains the pointwise evaluation of the manufactured, exact solution on the grid, and  $\mb v$ is the numerical solution vector, $h$ is the grid spacing and $d$ is the spatial dimension.

\subsection{Constant coefficients in one dimension}
We consider the diffusive viscous wave equation in one space dimension 
\[
u_{tt} - \alpha u_t - \beta^2 u_{xxt} - \gamma^2 u_{xx} = f,\quad x\in \Omega_{1d},\quad t\in [0, 0.5].
\]
Here, the second derivative in space $\partial^2/\partial x^2$ can be approximated by using the SBP operator $D_2^{(b)}$ with $b=1$, which is equivalent to the second derivative SBP operator with constant coefficient constructed in \cite{Mattsson2004}. 

First, we consider a similar problem to the example in Sec. 5.1 from \cite{Ling2023}, with material properties $\alpha=1$, $\beta=\gamma=0.1$, and manufactured solution $u=e^{-t}\cos (2\pi x)$. In this case, the forcing function $f$ is zero and the initial conditions are $u(x,0)=\cos (2\pi x)$ and $u_t(x,0)=-\cos (2\pi x)$. We choose the spatial domain $\Omega_{1d}= [0.1, 1.1]$ instead of $[0,1]$ to avoid special zero boundary data, and discretize by 81 grid points in space. 
\begin{figure}
    \centering
    \includegraphics[width=0.45\textwidth]{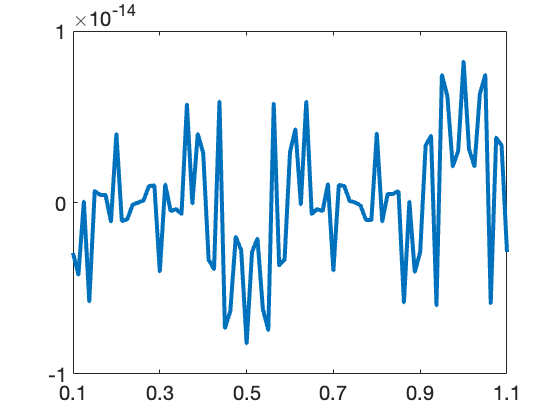}
    \includegraphics[width=0.45\textwidth]{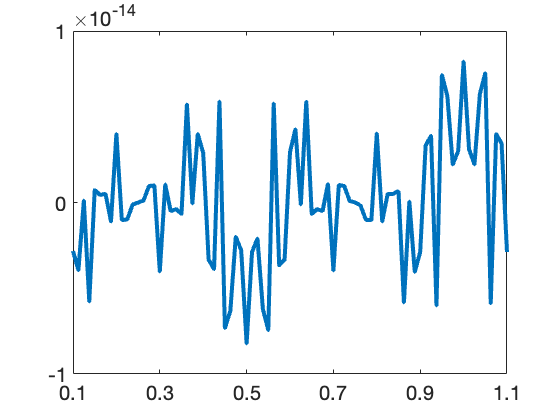}
    \caption{Error for the Dirichlet problem (left) and Neumann problem (right).}
    \label{fig:ex1}
\end{figure}

In Figure \ref{fig:ex1}, we plot the error at the final time for the problem with the Dirichlet boundary condition and the Neumann boundary condition. We observe that in both cases, the errors are zero to machine precision, and the numerical solutions are exact. This is a special case, because of the manufactured solution and material properties, the truncation errors of the spatial discretization cancel. More precisely, for the Dirichlet problem, the truncation errors are caused by the approximations of $\beta^2 u_{xxt}$ and $\gamma^2 u_{xx}$, and are in the form 
$T_{\beta}=\beta^2 Ch^{2p} \partial^{2p+2}/\partial x^{2p+2} u_t$ and $T_{\gamma}=\gamma^2 Ch^{2p} \partial^{2p+2}/\partial x^{2p+2} u$ on the interior grid points $\mathcal{N}_I$. On the same grid point, the constants $C$ in the truncation error $T_{\beta}$ and $T_{\gamma}$ are the same. With the manufactured solution $u=e^{-t}\cos (2\pi x)$, we have $u_t=-u$. Consequently, we have $T_{\beta}+T_{\gamma}=0$ when the coefficients are equal $\beta^2=\gamma^2$. In the same way, the truncation errors on the boundary grid points in $\mathcal{N}_B$ also cancel, resulting a spatial discretization without any truncation error. For the Neumann problem, there is an extra source of truncation error from the terms in \eqref{FN} imposing the boundary condition. For the same reason, we have $\FN=0$ when the manufactured solution satisfies $u_t=-u$ and the material parameters satisfy $\beta^2=\gamma^2$. We note that this error cancellation phenomenon does not occur in the LDG method shown in \cite{Ling2023}.

To verify convergence of our method, we choose a different manufactured solution so that the truncation errors do not cancel, $u=e^{-2t}\cos(2\pi x)$. In addition, we extend the time domain to $t\in [0,5]$ and allow sufficient time for the boundary truncation error to propagate into the interior. We consider both problems with Dirichlet boundary conditions and Neumann boundary conditions, and different combinations of material parameters corresponding to the cases in the error estimates. The initial data, boundary data and the forcing function are obtained through the manufactured solution.   

\begin{figure}
    \centering
    \includegraphics[width=0.45\textwidth]{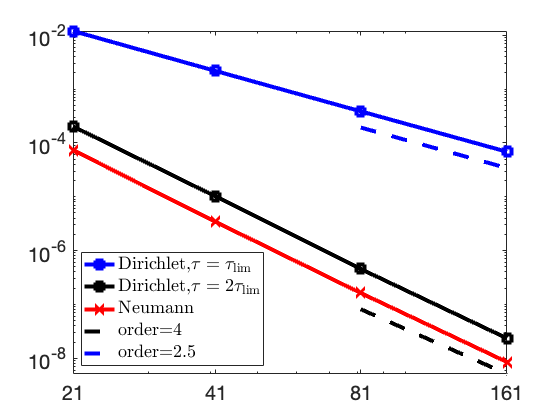}
    \includegraphics[width=0.45\textwidth]{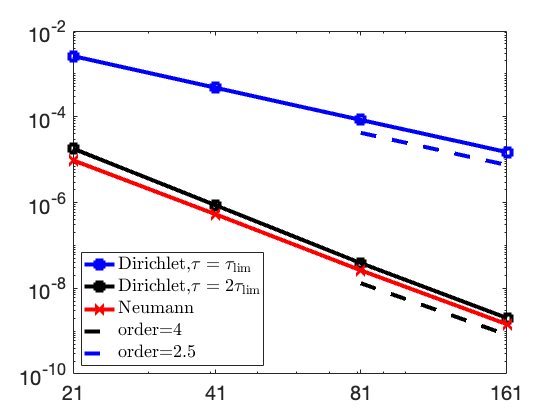}
    \includegraphics[width=0.45\textwidth]{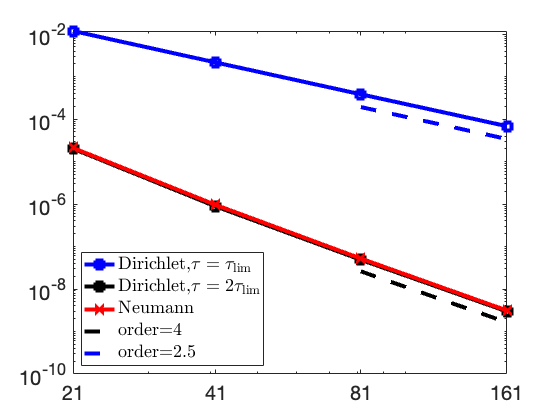}
    \includegraphics[width=0.45\textwidth]{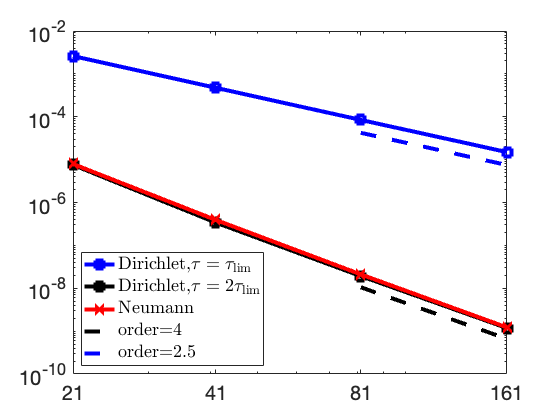}
    \caption{Error plots for Case 1 (top left), Case 2 (top right), Case 3 (bottom left) and Case 4 (bottom right).}
    \label{fig:ex1_1}
\end{figure}

In Figure \ref{fig:ex1_1}, We show the error plots for the four cases analyzed in Sec.~\ref{sec-err}, with material parameters defined as below 
\begin{align*}
    &\text{Case 1:}\ \alpha=\beta=0,\gamma=0.1,\quad \text{Case 2:}\ \alpha=1, \beta=0,\gamma=0.1,\\
    &\text{Case 3:}\ \alpha=0,\beta=\gamma=0.1,\quad \text{Case 4:}\ \alpha=1, \beta=0.1=\gamma=0.1.
\end{align*}
We observe that the numerical results agree very well with the error estimates derived in Sec.~\ref{sec-err}.

\subsection{Variable coefficients in one dimension}
We consider the diffusive viscous wave equation in one space dimension with variable coefficients, 
\[
u_{tt} - \alpha u_t - (\beta^2 u_{x})_{xt} - (\gamma^2 u_{x})_x = f,\quad x\in \Omega_{1d},\quad t\in [0, 0.5],
\]
\begin{figure}
    \centering
    \includegraphics[width=0.45\textwidth]{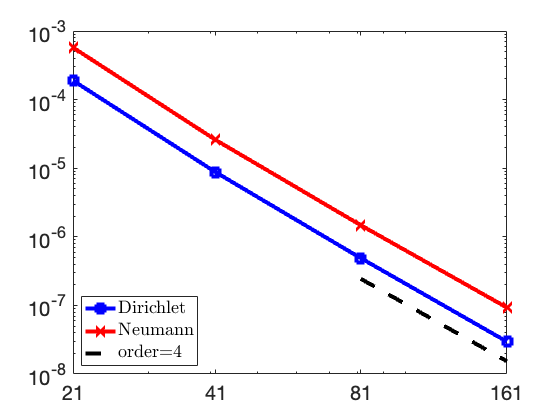}
    \includegraphics[width=0.45\textwidth]{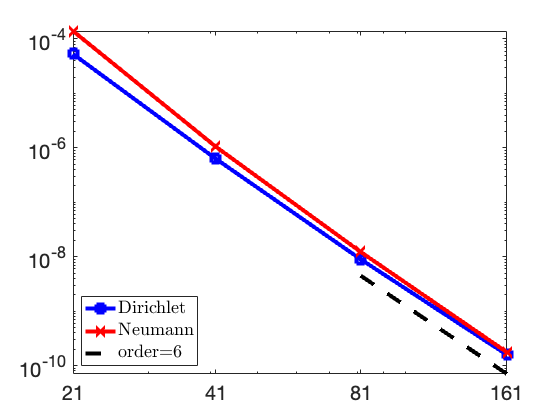}
    \caption{Error plots for the case with variable coefficients: the fourth order method (left) and the sixth order method (right).}
    \label{fig:ex2}
\end{figure}
with material properties 
\[
\alpha = e^{-x},\quad \beta = 0.2+0.1\sin(2\pi x),\quad \gamma = 0.15+0.1\sin(2\pi x).
\]
We choose the manufactured exact solution $u=e^{-2t}\cos(2\pi x)$ to compute initial and boundary data, and the forcing function $f$.

In Figure \ref{fig:ex2}, we plot the $l^2$ errors at the final time of the problem with Dirichlet or Neumann boundary conditions. For the Dirichlet boundary condition, the penalty parameters are chosen to be twice the limit by the stability requirement. With the fourth order accurate SBP operator with variable coefficient, the observed convergence rate is also fourth order, which aligns well with the error analysis for problems with constant coefficients.  In addition, we also used the sixth order accurate SBP operators and obtained a nearly sixth order convergence rate, shown in the error plot on the right side in Figure \ref{fig:ex2}. We note that this is half an order higher than the error analysis of the sixth order SBP operator for the wave equation with constant wave speed, whose convergence rate is 5.5 \cite{Wang2017}. 

\subsection{Two space dimension}
We consider the model problem \eqref{DVWE} in two space dimension $\Omega=[0,1]^2$, and a setup similar to numerical example 5 in \cite{Ling2023}. The material parameters are $\alpha=0$, $\beta=0.1$, and $\gamma = 0.4$. With both zero initial and boundary data, the solution is driven by forcing function $f$ in the form of a Ricker wavelet, 
\[
f = A (1-2\pi^2f_{re}^2 (t-0.1)^2) e^{-\pi^2f_{re}^2 (t-0.1)^2}.
\]

\begin{figure}
    \centering
    \includegraphics[width=0.45\textwidth]{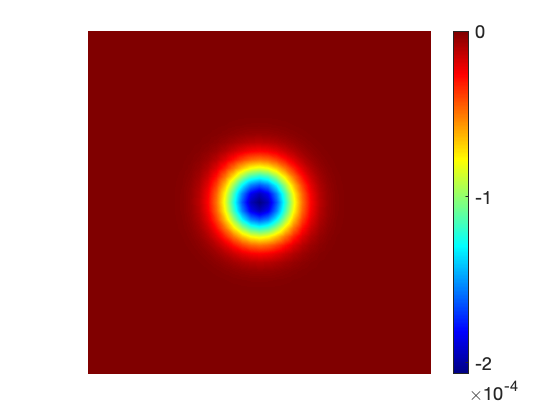}
    \includegraphics[width=0.45\textwidth]{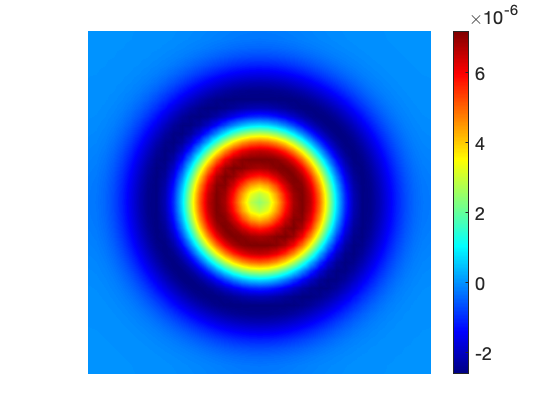}\\
    \includegraphics[width=0.45\textwidth]{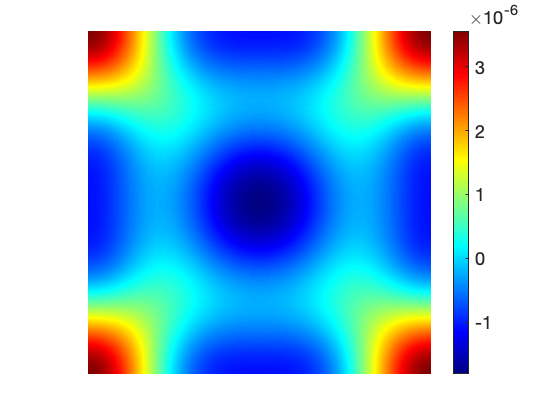}
    \includegraphics[width=0.45\textwidth]{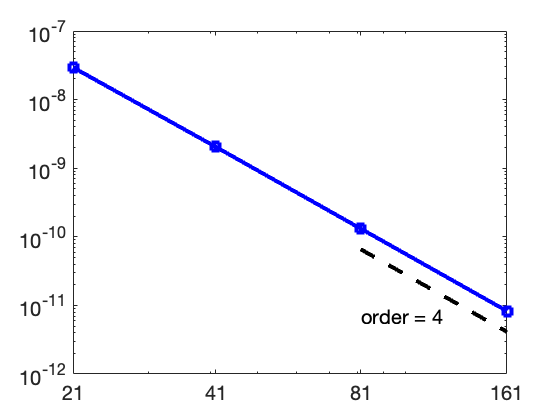}
    \caption{Solution plots for the case with Ricker wavelet computed with $41^2$ grid points at time $t=0.1$ (top left), $t=0.5$ (top right), and $t=2$ (bottom left); error plot for the solutions at $t=0.5$ (bottom right).}
    \label{fig:ex3}
\end{figure}

In numerical experiments, we take wavelet frequency $f_{re}=15$ and scaling factor $A=e^{-100((x-0.5)^2+(y-0.5)^2)}$ so that the peak is centered in the spatial domain and at time $t=0.1$.  We solve the governing equations with $N=41^2$ grid points and plot the solutions at three different time points $t=0.1, 0.5, 2$. In Figure \ref{fig:ex3}, we observe from the first two plots that the wave propagates symmetrically from the center of the domain where the peak is located. At $t=2$, the wave has ready reached the boundary and is reflected by the homogeneous Neumann boundary condition.  

Moreover, we have also carried out a convergence study. We use the fourth order accurate SBP operators for spatial discretization, and solve the governing equation until time $t=0.5$. We compute solutions with different mesh resolution of $21^2,41^2,81^2,161^2$ grid points. Since the analytical solution does not exist, we compute errors by using the reference solution computed on a very fine mesh with $641^2$ grid points. In Figure \ref{fig:ex3}, we observe clearly the optimal convergence rate of fourth order.

\section{Conclusion}
We have developed an SBP-SAT finite difference method for the diffusive viscous wave equation in second order form. Our approach ensures stability through the derivation of discrete energy estimates, and we have further provided error estimates for governing equations with constant coefficients and variable coefficients. The numerical experiments conducted have revealed optimal convergence rates for both cases.

For scenarios where the domain boundary or internal geometrical structure deviates from rectangular shapes, discretizing the governing equation on curvilinear grids, similar to the wave equation \cite{Virta2014}, can be a viable strategy. Additionally, in cases involving complex geometries, a hybrid method combining finite difference techniques on Cartesian grids with the discontinuous Galerkin method on unstructured grids has shown promise. Recent developments in this area, such as for the wave equation \cite{Wang2023}, inspire future work toward generalizing such hybrid methods for the diffusive viscous wave equation.

\bibliographystyle{plain} 
\bibliography{Siyang_References} 

\end{document}